\documentclass[11pt,reqno]{amsart}
\usepackage[vmargin=2cm,hmargin=2cm]{geometry}
\usepackage{amssymb, latexsym, amsfonts,amsmath,amsthm, color,etaremune,enumerate}
\usepackage{caption,float,textcomp,units,xfrac,url}
\title{Inequalities between overpartition ranks for all moduli}
\author{Alexandru Ciolan}
\address{Max-Planck-Institue f\"ur Mathematik, Vivatsgasse 7, 53111 Bonn}
\email{ciolan@mpim-bonn.mpg.de}
\newtheorem{Thm}{Theorem}

\newtheorem{Lem}{Lemma}

\theoremstyle{remark}

\newtheorem{rem}{Remark}

\newcommand{\cal}{\mathcal}

\newcommand{\bb}{\mathbb}

\DeclareMathOperator{\RE}{Re}
\DeclareMathOperator{\IM}{Im}
\makeatletter
\let\@@pmod\pmod
\DeclareRobustCommand{\pmod}{\@ifstar\@pmods\@@pmod}
\def\@pmods#1{\mkern4mu({\operator@font mod}\mkern 6mu#1)}
\makeatother

\newcommand{\N}{\mathbb{N}}
\renewcommand{\O}{\mathcal{O}}
\newcommand{\U}{\mathcal{U}}
\newcommand{\Z}{\mathbb{Z}}

\newcommand{\R}{\mathbb{R}}
\newcommand{\C}{\mathbb{C}}
\newcommand{\V}{\mathcal{V}}
\renewcommand{\l}{\ell}
\newcommand{\ol}{\overline}
\newcommand{\primesum}{\sideset{}{'}}
\numberwithin{equation}{section} 
\allowdisplaybreaks
\begin{document}
\keywords{Asymptotics, circle method, Dyson's rank, inequalities, Kloosterman sums, overpartitions}
\subjclass[2010]{11P72, 11P76, 11P82}
\begin{abstract}
In this paper we give a full description of the inequalities that can occur between overpartition ranks. If $ \ol N(a,c,n) $ denotes the number of overpartitions of $ n $ with rank congruent to $ a  $ modulo $ c,$  we prove that  for any $ c\ge7 $ and $ 0\le a<b\le\left\lfloor\frac{c}{2}\right\rfloor $ we have $ \ol N(a,c,n)>\ol N(b,c,n) $ for $n$ large enough.    That the sign of the rank differences $ \ol N(a,c,n)-\ol N(b,c,n) $ depends on the residue class of $ n $ modulo $ c $ in the case of small moduli, such as $ c=6, $ is known due to the work of Ji, Zhang and Zhao (2018) and Ciolan (2020). We show that the same behavior holds  for $ c\in\{2,3, 4,5\}. $  
\end{abstract}
\maketitle
\section{Introduction and statement of results}
\subsection{Dyson's rank}
In his attempt to find a combinatorial interpretation of the famous congruences 
\begin{align*}
p(5n+4)&\equiv  {0 \pmod*5,}\\
p(7n+5)&\equiv  {0 \pmod*7,}\\
p(11n+6)&\equiv  {0 \pmod*{11}}
\end{align*} 
discovered by Ramanujan \cite{Ram} for  $ p(n) $,  the number of partitions of $ n, $ Dyson \cite{Dyson} introduced the \textit{rank} of a partition, often referred to as \textit{Dyson's rank}, which is defined to be the largest part of the partition minus the number of its parts. As shown by Atkin and Swinnerton-Dyer \cite{ASD}, the rank would indeed explain the congruences modulo 5 and 7, but not those modulo 11. To justify the latter congruences, Dyson conjectured the existence of another partition statistic, called \textit{crank}. Some forty years later, Andrews and Garvan \cite{AG} found the right definition of the crank and proved that it simultaneously explains all the three congruences. 
The way in which Atkin and Swinnerton-Dyer \cite{ASD} proved Dyson's claim that the rank is equidistributed modulo 5 and 7, i.e., that for $ 0\le s\le 4 $ and $ 0\le t\le 6 $ we have
\begin{align*}
N(s,5,5n+4)&=\frac{p(5n+4)}{5},\\
N(t,7,7n+5)&=\frac{p(7n+5)}{7},
\end{align*} 
where $ N(a,c,n) $ denotes the number of partitions of $ n $ with rank congruent to $ a  $ modulo $ c, $ was by computing the generating functions associated to every rank difference $ N(s,\ell,\ell n+d)-N(t,\ell, \ell n+d), $ with $ 0\le d,s,t<\l, $ for $ \ell=5 $ and 7. While many of these turned out to be non-trivially zero, others were shown to be
infinite products or generalized Lambert series related to Ramanujan's third order
mock theta functions. Analogus, yet more technical formulas were  found by Atkin and Hussain \cite{AH} for $ \ell=11. $ 
\par It is not surprising that this generated a great amount of interest in studying rank differences and  rank inequalities for other moduli. In this regard, we have the inequalities 
\begin{eqnarray*}
& N(0,2,2n)  <  N(1,2,2n)&\text{if $n\ge1,$ }\\
& \phantom{2}N(0,4,n)  >  N(2,4,n)\phantom{2}&\text{if $n>26$ and $n\equiv0,1\pmod*4,$ }\\
& \phantom{2}N(0,4,n)  <  N(2,4,n)\phantom{2}&\text{if $n>26$ and $n\equiv2,3\pmod*4$ }
\end{eqnarray*}
found in \cite{AL} and \cite{Lewis} by Andrews and Lewis, who also conjectured that 
\begin{equation}\label{AndrewsLewisIneqs}
\begin{aligned}
 N(0,3,n)  <  N(1,3,n)&\quad\text{if $n\equiv0,2\pmod*3,$ }\\
 N(0,3,n)  >  N(1,3,n)&\quad\text{if $n\equiv1\phantom{,2}\pmod*3,$ }
\end{aligned}
\end{equation}
conjecture which was proven by Bringmann \cite{B} for  $ n\notin\{3,9,21\}, $ in which cases equality holds. Further, Bringmann and Kane \cite{BK} proved that, for any odd $ c>9 $ and for $ 0\le a<b\le\frac{c-1}{2}, $ we have $$ N(a,c,n)>N(b,c,n) $$ 
for $ n $ large enough. They also studied the sign of the rank differences $ N(a,c,n)-N(b,c,n)  $ for the moduli $ c=5,7 $ and 9, showing that this depends on the residue class of $ n $ modulo $ c. $ 
\subsection{Overpartitions}By an \textit{overpartition} of $ n $ we mean a partition in which the \textit{first} occurrence of a part may (or may not) be overlined, and by $ \ol p(n) $  we denote the number of overpartitions of $ n $. To illustrate with an example, we have $ p(4)=5, $ as the partitions of $ n=54  $ are given by $$4,~3+1,~2+2,~2+1+1,~1+1+1+1,$$
whereas $ \ol p(n)=14, $ as the overpartitions of $ n=4 $ are $$4,~\overline 4,~3 + 1,~\overline 3+ 1,~3 +\overline 1,~\overline 3 + \overline 1,~2+ 2,~\ol 2 + 2,~2+ 1+1,~\ol2 + 1 + 1,~2+ \ol1 + 1,~\ol2 +\ol 1 + 1,~1+1 + 1 + 1,~\ol 1+1 + 1 + 1.$$ \par The rank (also called $ D $-rank) of an overpartition is defined in exactly the same way as for partitions. We denote by $ \ol N(m,n) $ the number of overpartitions of $ n $ with rank $ m, $ and by $ \ol N(a,c,n) $ the number of overpartitions of $ n $ with rank congruent to $ a $ modulo $ c. $ 
\par As opposed to Ramanujan's congruences for the partition function, in the case of overpartitions there are no congruences of the form $ \ol p(\ell n+d)\equiv0\pmod*{\l} $ for primes $ \l\ge3. $ Therefore, the rank differences $\ol N(s,\ell,\ell n+d)-\ol N(t,\ell,\ell n+d)$ provide a measure of the extent to which the
rank fails to produce such a congruence. As such, trying to find the associated generating functions and, whenever possible, the sign of these rank differences and the resulting inequalities, has turned into a vivid area of recent research. 
\subsection{Motivation} Lovejoy and Osburn \cite{LO} found formulas for the rank differences $\ol N(s,\ell,\ell n+d)-\ol N(t,\ell,\ell n+d)$ in terms of modular functions and generalized Lambert series for $ \ell=3 $ and $ \l=5, $ whereas Jennings-Shaffer \cite{CJS} computed the rank differences for $ \ell=7 $ using the result of Bringmann and Lovejoy \cite{BJoy} that the overpartition rank function is the holomorphic part of a harmonic Maass form. More recently, Cui, Gu and Su \cite{CGS} computed the rank differences for $ \ell=4 $ and $ \l=8,  $ obtaining a few identities and inequalities that were also independently formulated and proven in the current paper, while  Ji, Zhang and Zhao \cite{JZZ}, and Wei and Zhang \cite{WZ} computed the rank differences for $ \l=6 $ and $ \l=10 $ by relating them to Ramanjuan's third and tenth order mock theta functions. They also established a few inequalities and left several others as conjectures, all of which were subsequently proven, with different methods, by the author \cite{Ciolan}.
Some of them are simple inequalities between ranks, similar to \eqref{AndrewsLewisIneqs}, such as
\begin{eqnarray*}
& \ol N(0,3,n)  >  \ol N(1,3,n)&\text{if $n\equiv0,1\pmod*3,$ }\\
& \ol N(0,3,n)  <  \ol N(1,3,n)&\text{if $n\equiv2\phantom{,1}\pmod*3,$ }
\end{eqnarray*} 
while others involve sums of ranks, e.g., 
\begin{align*}
\overline N(0,6,n)+\overline N(1,6,n)&> \overline N(2,6,n)+\overline N(3,6,n),\\
\overline N(0,10,n)+\overline N(1,10,n)&> \overline N(4,10,n)+\overline N(5,10,n).\end{align*}
\subsection{Main results} To the best of our knowledge, no such inequalities have been found for moduli other than $ \ell\in\{6,10 \}$ and, in some particular cases,  $ \ell\in\{4,8\}. $ In this paper we give a complete characterization of the  inequalities of the form  
\begin{equation}
\label{GenericIneq}
\ol N(a,c,n)>\ol N(b,c,n),
\end{equation} with $ 0\le a<b\le\left\lfloor\frac{c}{2}\right\rfloor, $
for all moduli $ c\ge2. $  This is inspired, partly, by the results of Bringmann and Kane \cite{BK} on inequalities between partition ranks, and comes to complete the work initiated in \cite{Ciolan} by the author. In contrast to \cite{BK} however, where only the case $ 2\nmid c $ was dealt with, in the case of partition ranks $ N(a,c,n) $, here we are also able to treat the case when $ c $ is even.
\par As we will see in Theorem \ref{ThmGeneralIneqs}, the inequality \eqref{GenericIneq} holds for all $ c\ge7 $ and $ n $ large enough, while for $ 2\le c\le 5,$ the only cases that were not treated by now and which we study in Theorems \ref{ThmMod5}--\ref{ThmMod4}, the sign of the inequality changes with the residue class of $ n $ modulo $ c. $ In addition to the inequalities, we prove that some interesting patterns and identities hold for $ c\in\{2,4\}. $  \par If one might perhaps expect that the higher moduli are influenced by their smallest prime divisors, which would then determine the sign of the rank difference, we will show that this is not the case and we will explain the reason for which the inequalities are not affected by the residue class modulo $ c $ for $  c\ge7 $  and $ n $ large enough. More precisely, we prove the following.
\begin{Thm}
\label{ThmGeneralIneqs}
If $ c\ge7, $ there exists $ n_{a,b,c} $ depending on $ a,b,c $ such that 
\[\ol N(a,c,n)>\ol N(b,c,n)\]
for $ 0\le a<b\le\left\lfloor\frac{c}{2}\right\rfloor $ and for any  $n>n_{a,b,c}.$ 
\end{Thm}
For the only moduli left to study, namely $ c\in \{ 2,3,4,5\}, $ we prove that the following results hold. The reader interested in the inequalities and identities proven for $ c= 6 $ can consult \cite{Ciolan} and \cite{JZZ}.
\begin{Thm}
\label{ThmMod5}
Apart from a few exceptions,  we have 
\begin{align*}
\ol N(0,5,n)&>\ol N(1,5,n)~ \text{if~$n\equiv0,1,3\pmod*5,$}\\
\ol N(0,5,n)&<\ol N(1,5,n)~ \text{if~$n\equiv2,4\phantom{,1}\pmod*5,$} \\
\ol N(1,5,n)&>\ol N(2,5,n)~ \text{if~$n\equiv0,2,4\pmod*5,$}\\
\ol N(1,5,n)&<\ol N(2,5,n)~ \text{if~$n\equiv1,3\phantom{,1}\pmod*5,$}\\
\ol N(0,5,n)&>\ol N(2,5,n)~ \text{if~$n\equiv0,1,3\pmod*5,$}\\
\ol N(0,5,n)&<\ol N(2,5,n)~ \text{if~$n\equiv4\phantom{,1,2}\pmod*5,$}\\
\ol N(0,5,n)&=\ol N(2,5,n)~ \text{if~$n\equiv2\phantom{,1,2}\pmod*5.$}
\end{align*}
\end{Thm}
The values of $ n  $ for which the inequalities stated in Theorem \ref{ThmMod5} do not hold are given in the following table, accompanied by the corresponding exceptions.\medskip\medskip
\begin{table}[h]
\begin{tabular}{|c|c|}\hline
$ n $ & Exceptions\\\hline\hline
$ 5,35$ & $ N(0,5,n)<N(1,5,n) $\\
$ 8,20,25,40$ & $ N(0,5,n)=N(1,5,n) $\\
$ 2$ & $ N(1,5,n)>N(2,5,n) $\\
$ 1,10,11,15,30$ & $ N(1,5,n)=N(2,5,n) $\\
$ 4,5,8,9,24$ & $ N(0,5,n)=N(2,5,n) $\\\hline
\end{tabular}\medskip
\caption{Exceptions for $ c=5 $}\label{TableMod5}
\end{table}
\begin{Thm}
\label{ThmMod2}
For  $ n\ge1  $ we have
\begin{align*}
\ol N(0,2,n)&>\ol N(1,2,n)~ \text{if~$n$~is odd},\\
\ol N(0,2,n)&<\ol N(1,2,n)~ \text{if~$n$ is even}.
\end{align*}
\end{Thm}
\begin{Thm}
\label{ThmMod3}
For  $ n\ge1  $ we have
\begin{align*}
\ol N(0,3,n)&>\ol N(1,3,n)~ \text{if~}n\equiv 0,1\pmod*3,\\
\ol N(0,3,n)&<\ol N(1,3,n)~ \text{if~}n\equiv 2\phantom{,1}\pmod*3.
\end{align*}
\end{Thm}
By computing rank differences for $ \ell=4 $ and $ \l=8 $, Cui, Gu and Su \cite{CGS} established very recently\footnote{The author only became aware of  \cite{CGS} after the completion of the present paper, and shortly before its submission. The overlapping results are, therefore, to be seen as independent of one another.} a few identities and inequalities, see Theorems 1.2--1.5 in \cite{CGS}, most of which also follow from Theorem \ref{ThmGeneralIneqs} and the next result. However, while Theorem 1.5 in \cite{CGS} gives several inequalities that hold modulo 8 for $ n $ in certain residue classes, it does not capture the full behavior of the inequalities.  This is answered by Theorem \ref{ThmGeneralIneqs} of the current paper, applied to the case $ c=8. $  
\begin{Thm}\label{ThmMod4}
For $ n\ge1  $ we have 
\begin{align*}
\ol N(0,4,n)&>\ol N(1,4,n)~ \text{if~$n$~is odd},\\
\ol N(0,4,n)&<\ol N(1,4,n)~\text{if~$n$~is even}, \\
\ol N(1,4,n)&<\ol N(2,4,n)~ \text{if~$n$~is odd},\\
\ol N(1,4,n)&>\ol N(2,4,n)~ \text{if~$n$~is even},\\
\end{align*}and
\[\ol N(0,4,n)-\ol N(2,4,n)=\begin{cases} 0 & \text{if~$n$ is not a square},\\
2 & \text{if~$n$ is  a square.}\end{cases} 
\]
\end{Thm}

\begin{rem}
The number $ n_{a,b,c} $ depends only on $ a,b,c $ and can be found after a finite computation. We will make this precise in Section \ref{Proof}. 
\end{rem}
\begin{rem}
The inequalities from Theorems \ref{ThmGeneralIneqs}, \ref{ThmMod5} and \ref{ThmMod3} also shed light on the signs of the coefficients of the rank differences found for $ \ell\in\{3,4,5,7,8\} $ in \cite{CGS}, \cite{CJS} and \cite{LO}. A study of rank inequalities based on $ q $-series  expansions  was done in \cite{JZZ} for $\l\in\{6,10\}, $ but this was only possible in the case of some fairly simple  expressions for which  it is not difficult to conclude, say, that the coefficients are all positive (see, e.g., the proof of \cite[Theorem 1.4]{JZZ}). However, this is also the reason for which other inequalities cannot be proven with that approach. For $ \l\in\{3,4,5,7,8\}, $ the rank differences are written as sums of various quotients of infinite products, and the sign of the coefficients of these rather complicated $ q $-series expansions (see Theorems 1.3--1.4 of \cite{CGS}, Theorem 1.1 of \cite{CJS}, or Theorems 1.1--1.2 in \cite{LO}) can generally not be guessed a priori.
\begin{rem}
The identity $ \ol N(0,5,n)=\ol N(2,5,n) $ for $ n\equiv2\pmod*5 $ does not follow from our theorem, but it was already proven by Lovejoy and Osburn, see eq. (12) from \cite[Theorem 1.2]{LO}.  Nevertheless, we include it here for completeness.
\end{rem} 
\end{rem}
\begin{rem}
There is no way to deduce any of the inequalities listed in Theorem \ref{ThmMod5} simply by adding, say, the inequalities between ranks modulo 10 obtained in \cite[Theorem 2]{Ciolan}. In fact, this is impossible in general, precisely because of \eqref{GenericIneq} and the fact that $ \ol N(a,c,n)=\ol N(c-a,c,n). $ 
\end{rem}
\begin{rem}
Theorem \ref{ThmMod2} shows that the sign of the rank difference $ \ol N(0,2,n)-\ol N(1,2,n) $ alternates with the parity of $ n. $  A similar result holds (see \cite{Ciolan2}) for $ p_r(0,2,n)-p_r(1,2,n), $ where $ p_r(a,m,n) $ denotes the number of partitions of $ n $ into $ r $-th powers that have a number of parts congruent to $ a $ modulo $ m. $
\end{rem}
\subsection{Overview} Our work relies heavily on the results established in \cite{Ciolan}, and the general approach  that we follow is, to some extent, similar to what was done in the case of partition ranks by Bringmann and Kane \cite{BK}. Therefore, although the current paper is self-contained and may be read independently, the reader is warmly invited to consult these two references. \par The paper is structured as follows. In Section \ref{Strategy} we describe the main ideas of the proof of Theorem \ref{ThmGeneralIneqs}, which we give in Section \ref{Proof}, together with the proofs of Theorems \ref{ThmMod5}--\ref{ThmMod4}. In order to explain the general strategy, we need to introduce some notation  and recall the results of \cite{Ciolan},  and this we do in Section \ref{Preliminaries}.

\section{Preliminaries}\label{Preliminaries}
\subsection{Rank generating functions}\label{RankGF} Before being able to fully explain our approach, we need a few preparatory steps. We begin by recalling that, if $ q=e^{2\pi iz}, $ with $ z\in\C$ and $\IM(z)>0, $  the overpartition generating function (see, e.g., \cite{CL}) is given by 
\begin{equation*}
\overline P(q):=\sum_{n\ge 0}\overline{p}(n)q^n=\frac{\eta(2z)}{\eta^2(z)}=\prod_{n=1}^{\infty}\frac{1+q^n}{1-q^n}=1+2q+4q^2+8q^3+14q^4+\cdots,
\end{equation*}
where $$ \eta(z):=q^{\frac1{24}}\prod_{n=1}^{\infty}(1-q^n) $$ stands, as usual, for the Dedekind eta function. Further, we know from \cite{Lovejoy} that
\begin{align*}
	\O(u;q):=\sum_{n=0}^{\infty}\sum_{m=-\infty}^{\infty}\overline{N}(m,n)u^mq^n&=\sum_{n=1}^{\infty}\frac{(-1)_nq^{\frac12n(n+1)}}{(uq,q/u)_n}\nonumber\\
	&=\frac{(-q)_{\infty}}{(q)_{\infty}}\left(1+2\sum_{n\ge1}\frac{(1-u)(1-u^{-1})(-1)^nq^{n^2+n}}{(1-uq^n)(1-u^{-1}q^n)} \right),
\end{align*}
where for $ a,b\in\C $ and $ n\in\N\cup\{\infty\} $ we use the $ q $-Pochhammer symbols
\begin{align*}
(a)_n&:=\prod_{r=0}^{n-1}(1-aq^r),\\
(a,b)_n&:=\prod_{r=0}^{n-1}(1-aq^r)(1-bq^r).
\end{align*}
\par For $ 0<a<c $ coprime positive integers and $ \zeta_n=e^{\frac{2\pi i}{n}} $ the standard primitive $ n $-th root of unity, let
\begin{equation}\label{OA}
\cal O \left( \frac ac;q\right) :=\cal O \left(\zeta_c^a;q \right)=1+\sum_{n=1}^{\infty} A\left(\frac ac;n \right)q^n.
\end{equation}
\par The letters $ h $ and $ k $ will denote throughout coprime positive integers, with $ 0\le h<k. $ If $ k=1, $ we set $ h=0, $ this being the only instance when $ h=0 $ is allowed. Let $  \widetilde{k}=0 $ if $k$ is even, and $ \widetilde{k}=1 $ if $ k $ is odd, and put $ k_1=\frac{k}{(c,k)}, $ $ c_1=\frac c{(c,k)}. $ Further, let $ 0\le \ell <c_1 $ and $ h'\in\mathbb Z $ be given by the congruences $ \ell\equiv ak_1\pmod*{c_1}, $ respectively  $ hh'\equiv-1\pmod* k. $

\subsection{Kloosterman sums}  Recall that if \[((x)):=\begin{cases}
x-\lfloor x\rfloor-\frac12 & \text{if~}x\in\R\setminus\Z,\\
0 & \text{if~}x\in\Z,
\end{cases}\] and 
\[S_{h,k}:=\sum_{\mu=0}^{k-1}\left( \left( \frac{\mu}{k}\right) \right)\left( \left(\frac{h\mu}{k} \right) \right)\] is the  so-called \textit{Dedekind sum},
then \[\omega_{h,k}:=\exp( \pi i S_{h,k}  ) \]
is the multiplier that appears in the transformation of the partition function. 
\par In what follows, we will make use of several  Kloosterman-type sums, which we define below. Here and throughout, by $ \sum'_h $ we always indicate summation over the integers  $ 0\le h<k $ that are coprime to $ k. $ 
\par If $ c\mid k, $ let
\begin{equation*}
A_{a,c,k}(n,m):=(-1)^{k_1+1}\tan\left(\frac{\pi a}{c} \right) \primesum\sum_h\frac{\omega^2_{h,k}}{\omega_{h,k/2}}\cdot \cot\left(\frac{\pi ah'}{c} \right) \cdot e^{-\frac{2\pi ih'a^2k_1}{c}}\cdot e^{\frac{2\pi i}{k}(nh+mh')}, 
\end{equation*} and
\begin{equation*}
B_{a,c,k}(n,m):=-\frac{1}{\sqrt2}\tan\left(\frac{\pi a}{c} \right) \primesum\sum_h\frac{\omega^2_{h,k}}{\omega_{2h,k}}\cdot \frac{1}{\sin\left(\frac{\pi ah'}{c} \right)}\cdot e^{-\frac{2\pi ih'a^2k_1}{c}}\cdot e^{\frac{2\pi i}{k}(nh+mh')}.
\end{equation*} 
\par If $ c\nmid k $ and $ 0<\frac{\ell}{c_1}\le\frac14,$   let 
\begin{equation*}
D_{a,c,k}(n,m):=\frac1{\sqrt2}\tan\left(\frac{\pi a}{c} \right)\primesum\sum_h\frac{\omega^2_{h,k}}{\omega_{2h,k}}\cdot  e^{\frac{2\pi i}{k}(nh+mh')},
\end{equation*} 
and if $ c\nmid k $ and $ \frac34<\frac{\ell}{c_1}<1,$ let 
\begin{equation*}
D_{a,c,k}(n,m):=-\frac1{\sqrt2}\tan\left(\frac{\pi a}{c} \right)\primesum\sum_h\frac{\omega^2_{h,k}}{\omega_{2h,k}}\cdot  e^{\frac{2\pi i}{k}(nh+mh')}.
\end{equation*}
\par Finally, if $ c\nmid k, $ set
\begin{equation*}
\delta_{c,k,r}:=\begin{cases}
 \frac{1}{16}-\frac{ \ell}{2c_1}+\frac{\ell^2}{c_1^2}-r\frac{\ell}{c_1} &\text{{if~}}0<\frac{\ell}{c_1}\le\frac14,\\
0 & \text{if~}\frac14<\frac{\ell}{c_1}\le\frac34,\\
 \frac{1}{16}-\frac{3\ell}{2c_1}+\frac{\ell^2}{c_1^2}+\frac12-r\left(1-\frac{\ell}{c_1} \right)  & \text{if~}\frac34<\frac{\ell}{c_1}<1,
\end{cases}
\end{equation*}
%
and
\begin{equation*}
m_{a,c,k,r}:=\begin{cases}
-\frac{1}{2c_1^2}( 2(ak_1-\l)^2+c_1(ak_1-\ell) +2rc_1(ak_1-\ell)) &\text{{if~}}0<\frac{\ell}{c_1}\le\frac14,\\
0 & \text{if~}\frac14<\frac{\ell}{c_1}\le\frac34,\\
-\frac{1}{2c_1^2}( 2(ak_1-\l)^2 +3c_1(ak_1-\ell )-2rc_1(ak_1-\l)-c_1^2(2r-1)) & \text{if~}\frac34<\frac{\ell}{c_1}<1.
\end{cases}
\end{equation*}

\subsection{Modular transformations} If $ \frac bc\in(0,1)\setminus\left\{\frac12\right\}, $  define
\[s(b,c):=\begin{cases}
0 & \text{if~}0<\frac bc\le\frac14,\\
1 & \text{if~}\frac14<\frac bc\le\frac{3}{4},\\
2 & \text{if~}\frac{3}{4}<\frac bc<1,
\end{cases}
\quad\text{and}\quad
t(b,c):=\begin{cases}
1 & \text{if~}0<\frac bc<\frac12,\\
3 & \text{if~}\frac12<\frac bc<1.
\end{cases}
\]
For reasons of space, we will write $ s=s(b,c) $ and $ t=t(b,c). $ If $ 0<a<c $ are coprime with $ c>2, $  let
\begin{align*}\U\left(\frac ac;q \right)=\U\left(\frac ac;z \right)&:=\frac{\eta\left(\frac z2 \right) }{\eta^2(z)}\sin\left(\frac{\pi a}{c} \right)\sum_{n\in\Z}\frac{(1+q^n)q^{n^2+\frac n2}}{1-2\cos\left(\frac{2\pi a}{c} \right)q^n+q^{2n} },   \\
\U(a,b,c;q)=\U(a,b,c;z)&:=\frac{\eta\left(\frac z2 \right) }{\eta^2(z)} e^{\frac{\pi ia}{c}\left(\frac{4b}{c}-1-2s \right)}q^{\frac{sb}{c}+\frac{b}{2c}-\frac{b^2}{c^2}}\sum_{m\in\Z}\frac{q^{\frac{m}{2}(2m+1)+ms}}{1-e^{-\frac{2\pi ia}{c}}q^{m+\frac bc}},\\
\cal V\left(a,b,c; q\right)=\V(a,b,c;z)&:=\frac{\eta\left(\frac z2 \right) }{\eta^2(z)}e^{\frac{\pi i a}{c}\left(\frac{4b}{c}-1-2s \right) }q^{\frac{sb}{c}+\frac{b}{2c}-\frac{b^2}{c^2}}\sum_{m\in\bb Z}\frac{q^{\frac{m(2m+1)}{2}+ms}\left( 1+e^{-\frac{2\pi i a}{c}}q^{m+\frac bc}\right)  }{1-e^{-\frac{2\pi i a}{c}}q^{m+\frac bc}},\\
\cal O\left(a,b,c; q\right)=\O(a,b,c;z)&:=\frac{\eta(2z)}{\eta^2(z)}e^{\frac{\pi i a}{c}\left(\frac{4b}{c}-1-t \right) }q^{\frac{tb}{2c}+\frac b{2c}-\frac{b^2}{c^2}}\sum_{m\in\Z}(-1)^m\frac{q^{\frac{m}{2}(2m+1)+\frac{mt}{2}} }{1-e^{-\frac{2\pi i a}{c}}q^{m+\frac bc}},  \\
\V\left(\frac ac; q\right)=\V\left(\frac ac;z \right)&:= \frac{\eta(2z)}{\eta^2(z)}q^{\frac14}\sum_{m\in\Z}\frac{q^{m^2+m}\left( 1+e^{-\frac{2\pi ia}{c}}q^{m+\frac12}\right) }{1-e^{-\frac{2\pi ia}{c}}q^{m+\frac12}},\end{align*}
and consider the Mordell-type integral 
\[I_{a,c,k,\nu}:=\int_{\R}e^{-\frac{2\pi zx^2}{k}}H_{a,c}\left( \frac{2\pi i\nu}{k}-\frac{2\pi zx}{k}-\frac{\widetilde{k}\pi i}{2k} \right) dx, \] where $\nu\in\mathbb Z,$ $ k\in\bb N $ and $ \widetilde k $ are as defined in Section \ref{RankGF}, and \begin{equation*}
H_{a,c}(x):=\frac{e^x}{1-2\cos\left(\frac{2\pi a}{c} \right)e^x+e^{2x} }.
\end{equation*}
Using Poisson summation, Bringmann and Lovejoy \cite{BJoy} proved the following transformation laws.
\begin{Thm}[{\cite{BJoy}}]\label{Transformations}
Assume the notation above and let $ q=e^{\frac{2\pi i}{k}(h+iz)} $ and $ q_1=e^{\frac{2\pi i}{k}\left( h'+\frac iz\right) }, $ with $ z\in\C  $ and $ \RE (z)>0. $ 
\begin{enumerate}[{\rm(1)}]
\item If $ c\mid k $ and $2\mid k, $  then 
\begin{multline*}\O\left(\frac ac;q \right)=(-1)^{k_1+1}i\cdot e^{-\frac{2\pi a^2h'k_1}{c}}\cdot\tan\left(\frac{\pi a}{c} \right)\cdot\cot\left( \frac{\pi ah'}{c}\right)\frac{\omega^2_{h,k}}{\omega_{h,k/2}}z^{-\frac12}\cdot \O\left(\frac{ah'}{c};q_1 \right) \\ +  \frac{4\sin^2\left(\frac{\pi a}{c}\right)\cdot\omega^2_{h,k}}{\omega_{h,k/2}\cdot k}z^{-\frac12}\sum_{\nu=0}^{k-1}(-1)^{\nu}e^{-\frac{2\pi ih'\nu^2}{k}}\cdot I_{a,c,k,\nu}(z) .\end{multline*}
\item If $ c\mid k $ and $ 2\nmid k, $  then 
\begin{multline*}\O\left(\frac ac;q \right)=-\sqrt 2i\cdot e^{\frac{\pi ih'}{8k}-\frac{2\pi i a^2h'k_1}{c}}\cdot\tan\left(\frac{\pi a}{c} \right)\frac{\omega^2_{h,k}}{\omega_{2h,k}}z^{-\frac12}\cdot \U\left(\frac{ah'}{c};q_1 \right) \\ +  \frac{4\sqrt2\sin^2\left(\frac{\pi a}{c}\right)\cdot\omega^2_{h,k}}{\omega_{2h,k}\cdot k}z^{-\frac12}\sum_{\nu=0}^{k-1}e^{-\frac{\pi ih'}{k}(2\nu^2-\nu)}\cdot I_{a,c,k,\nu}(z) .\end{multline*}
\item If $ c\nmid k, $ $ 2\mid k $  and $ c_1\ne 2, $ then  
\begin{multline*}\O\left(\frac ac;q \right)=2 e^{-\frac{2\pi ia^2h'k_1}{c_1c}}\cdot\tan\left(\frac{\pi a}{c} \right)\frac{\omega^2_{h,k}}{\omega_{h,k/2}}z^{-\frac12}\cdot (-1)^{c_1(\ell+k_1)}\cdot \O\left(ah',\frac{\ell c}{c_1},c;q_1 \right) \\ +  \frac{4\sin^2\left(\frac{\pi a}{c}\right)\cdot\omega^2_{h,k}}{\omega_{h,k/2}\cdot k} z^{\frac12}\sum_{\nu=0}^{k-1}(-1)^{\nu}e^{-\frac{2\pi ih'\nu^2}{k}}\cdot I_{a,c,k,\nu}(z) .\end{multline*}
\item If $ c\nmid k, $ $ 2\mid k $  and $ c_1=2, $ then 
\begin{multline*}\O\left(\frac ac;q \right)= e^{-\frac{\pi ia^2h'k_1}{c}}\cdot\tan\left(\frac{\pi a}{c} \right)\frac{\omega^2_{h,k}}{\omega_{h,k/2}\cdot k}z^{-\frac12}\cdot \V\left(\frac{ah'}{c};q_1 \right) \\ +  \frac{4\sin^2\left(\frac{\pi a}{c}\right)\cdot\omega^2_{h,k}}{\omega_{h,k/2}\cdot k} z^{\frac12}\sum_{\nu=0}^{k-1}(-1)^{\nu}e^{-\frac{2\pi ih'\nu^2}{k}}\cdot I_{a,c,k,\nu}(z) .\end{multline*}
\item If $ c\nmid k, $ $ 2\nmid k $  and $ c_1\ne 4, $ then 
\begin{multline*}\O\left(\frac ac;q \right)=\sqrt 2 e^{\frac{\pi ih'}{8k}-\frac{2\pi i a^2h'k_1}{c_1c}}\cdot\tan\left(\frac{\pi a}{c} \right)\frac{\omega^2_{h,k}}{\omega_{2h,k}}z^{-\frac12}\cdot \U\left(ah',\frac{\ell c}{c_1},c;q_1 \right) \\ +  \frac{4\sqrt2\sin^2\left(\frac{\pi a}{c}\right)\cdot\omega^2_{h,k}}{\omega_{2h,k}\cdot k} z^{\frac12}\sum_{\nu=0}^{k-1}e^{-\frac{\pi ih'}{k}(2\nu^2-\nu)}\cdot I_{a,c,k,\nu}(z) .\end{multline*}
\item If $ c\nmid k, $ $ 2\nmid k $  and $ c_1= 4, $ then 
\begin{multline*}\O\left(\frac ac;q \right)= e^{\frac{\pi ih'}{8k}-\frac{2\pi i a^2h'k_1}{c_1c}}\cdot\tan\left(\frac{\pi a}{c} \right)\frac{\omega^2_{h,k}}{\sqrt2\cdot\omega_{2h,k}}z^{-\frac12}\cdot \V\left(ah',\frac{\ell c}{c_1},c;q_1 \right) \\ +  \frac{4\sqrt2\sin^2\left(\frac{\pi a}{c}\right)\cdot\omega^2_{h,k}}{\omega_{2h,k}\cdot k} z^{\frac12}\sum_{\nu=0}^{k-1}e^{-\frac{\pi ih'}{k}(2\nu^2-\nu)}\cdot I_{a,c,k,\nu}(z) .\end{multline*}
\end{enumerate}
\end{Thm}

\subsection{Circle Method}\label{SubsectionCircleMethod}  By Cauchy's Theorem, for any $ n\ge1 $ we have
\[A\left(\frac ac;n \right)=\frac1{2\pi i}\int_{\cal C}\frac{\O\left(\frac ac;q \right)}{q^{n+1}}dq,\]
where $\cal C $ is the circle of radius $ e^{-\frac{2\pi}{n}} $ parametrized by $ q=e^{-\frac{2\pi}{n}+2\pi it}, $  $ t\in[0,1], $ from where we further get
\[A\left(\frac ac;n \right)=\int_0^1\O\left(\frac ac;e^{-\frac{2\pi}{n}+2\pi it} \right)\cdot e^{2\pi-2\pi int}dt.\]
If $ \frac{h_1}{k_1}<\frac hk<\frac{h_2}{k_2} $ are adjacent Farey fractions in the Farey sequence of order $ N=\lfloor\sqrt n\rfloor, $ we set \[\vartheta_{h,k}':=\frac1{k(k_1+k)}\quad\text{and}\quad\vartheta_{h,k}'':=\frac1{k(k_2+k)}.\]
Splitting the path of integration along the Farey arcs $ -\vartheta_{h,k}'\le\Phi\le \vartheta_{h,k}'', $ where $ \Phi=t-\frac hk $ and $ 0\le h< k\le N $ with $(h,k)=1, $ we have
\begin{equation}\label{CircleFarey}
A\left(\frac ac;n \right)=\sum_{h,k}e^{-\frac{2\pi inh}{k}}\int_{-\vartheta_{h,k}'}^{\vartheta_{h,k}''} \cal O\left(\frac ac; e^{\frac{2\pi i}{k}(h+iz)}\right)\cdot e^{\frac{2\pi nz}{k}}  d\Phi,  \end{equation}
where $ z=\frac kn-k\Phi i. $ Applying Theorem \ref{Transformations} to \eqref{CircleFarey}, we can now express 
\begin{multline*}
A\left(\frac ac;n \right)=i\tan \left(\frac{\pi a}{c} \right) \sum_{\substack{h,k\\2|k,~c|k}}\frac{\omega^2_{h,k}}{\omega_{h,k/2}}(-1)^{k_1+1}\cot\left(\frac{\pi ah'}{c} \right)e^{-\frac{2\pi ia^2h'k_1}{c}-\frac{2\pi inh}{k}}\int_{-\vartheta'_{h,k}}^{\vartheta''_{h,k}} z^{-\frac12} e^{\frac{2\pi nz}{k}} \cal O\left(\frac{ah'}{c}; q_1\right)d\Phi\\
-\sqrt2i\tan \left(\frac{\pi a}{c} \right) \sum_{\substack{h,k\\2\nmid k,~c|k}}\frac{\omega^2_{h,k}}{\omega_{2h,k}}e^{\frac{\pi ih'}{8k}-\frac{2\pi ia^2h'k_1}{c}-\frac{2\pi inh}{k}}\int_{-\vartheta'_{h,k}}^{\vartheta''_{h,k}} z^{-\frac12} e^{\frac{2\pi nz}{k}} \cal U\left(\frac{ah'}{c}; q_1\right)d\Phi\\
+2\tan \left(\frac{\pi a}{c} \right) \sum_{\substack{h,k\\2|k,~c\nmid k,~c_1\ne 2}}\frac{\omega^2_{h,k}}{\omega_{h,k/2}}(-1)^{c_1(\ell+k_1)}e^{-\frac{2\pi ia^2h'k_1}{c_1c}-\frac{2\pi inh}{k}}\int_{-\vartheta'_{h,k}}^{\vartheta''_{h,k}} z^{-\frac12} e^{\frac{2\pi nz}{k}} \cal O\left(ah',\frac{\ell c}{c_1},c; q_1\right)d\Phi\\
+\tan \left(\frac{\pi a}{c} \right) \sum_{\substack{h,k\\2| k,~c\nmid k,~c_1= 2}}\frac{\omega^2_{h,k}}{\omega_{h,k/2}}e^{-\frac{\pi ia^2h'k_1}{c}-\frac{2\pi inh}{k}}\int_{-\vartheta'_{h,k}}^{\vartheta''_{h,k}} z^{-\frac12} e^{\frac{2\pi nz}{k}} \cal V\left(\frac{ah'}{c}; q_1\right)d\Phi\\
+\sqrt2\tan \left(\frac{\pi a}{c} \right) \sum_{\substack{h,k\\2\nmid k,~c\nmid k,~c_1\ne 4}}\frac{\omega^2_{h,k}}{\omega_{2h,k}}e^{\frac{\pi ih'}{8k}-\frac{2\pi ia^2h'k_1}{c_1c}-\frac{2\pi inh}{k}}\int_{-\vartheta'_{h,k}}^{\vartheta''_{h,k}} z^{-\frac12} e^{\frac{2\pi nz}{k}} \cal U\left(ah',\frac{\ell c}{c_1},c; q_1\right)d\Phi\\
+\frac1{\sqrt2}\tan \left(\frac{\pi a}{c} \right) \sum_{\substack{h,k\\2\nmid k,~c\nmid k,~c_1= 4}}\frac{\omega^2_{h,k}}{\omega_{h,k/2}}e^{\frac{\pi ih'}{8k}-\frac{2\pi ia^2h'k_1}{c_1c}-\frac{2\pi inh}{k}}\int_{-\vartheta'_{h,k}}^{\vartheta''_{h,k}} z^{-\frac12} e^{\frac{2\pi nz}{k}} \cal V\left(ah',\frac{\ell c}{c_1},c; q_1\right)d\Phi\\
+4\sin^2\left(\frac{\pi a}{c}\right) \sum_{\substack{h,k\\2|k}}\frac{\omega^2_{h,k}}{\omega_{h,k/2}\cdot k} e^{-\frac{2\pi inh}{k}} \sum_{\nu=0}^{k-1}(-1)^{\nu}e^{-\frac{2\pi ih'\nu^2}{k}}\int_{-\vartheta'_{h,k}}^{\vartheta''_{h,k}} z^{\frac12} e^{\frac{2\pi nz}{k}} I_{a,c,k,\nu}(z)d\Phi\\
+4\sqrt2\sin^2\left(\frac{\pi a}{c}\right) \sum_{\substack{h,k\\2\nmid k}}\frac{\omega^2_{h,k}}{\omega_{2h,k}\cdot k} e^{-\frac{2\pi inh}{k}} \sum_{\nu=0}^{k-1}e^{-\frac{\pi ih'}{k}(2\nu^2-\nu)}\int_{-\vartheta'_{h,k}}^{\vartheta''_{h,k}} z^{\frac12} e^{\frac{2\pi nz}{k}} I_{a,c,k,\nu}(z)d\Phi\\
=:\sum_1+\sum_2+\sum_3+\sum_4+\sum_5+\sum_6+\sum_7+\sum_8.
\end{multline*}

\subsection{Asymptotics for the coefficients $ A\left(\frac ac;n \right)  $} As explained in \cite{Ciolan}, a careful investigation shows that only the sums $ \sum_2 $ and $ \sum_5 $ contribute asymptotically, while all the others are seen to be of order $ O(n^{\varepsilon}), $ for any given $ \varepsilon>0. $ Consequently, we have the following.
\begin{Thm}[{\cite{Ciolan}}]\label{MainThm}
If $ 0<a<c $ are coprime positive integers and $ 
\varepsilon>0 $ is arbitrary, then  
\begin{align*}\label{thm}
A\left( \frac ac;n\right) &=i\sqrt{\frac 2n}\sum_{\substack{1\le k\le\sqrt n\\c|k,~2\nmid k}}\frac{B_{a,c,k}(-n,0)}{\sqrt k}\cdot\sinh\left( \frac{\pi\sqrt n}{k}\right)\nonumber \\
&\phantom{=~}+2\sqrt{\frac 2n}\sum_{\substack{1\le k\le \sqrt n\\c\nmid k,~2\nmid k,~c_1\ne4\\r\ge0,~ \delta_{c,k,r}>0}}\frac{D_{a,c,k}(-n,m_{a,c,k,r})}{\sqrt k}\cdot \sinh \left( \frac{4\pi\sqrt{\delta_{c,k,r}n}}{k} \right)+O_c(n^{\varepsilon}).
\end{align*}
\end{Thm}
As remarked in \cite[p. 5]{Ciolan}, in evaluating the sums $ B_{a,c,k} $ and $ D_{a,c,k} $ from Theorem \ref{MainThm}, the integer $ h' $ is always taken to be even, cf. Bringmann and Lovejoy \cite[pp. 14--15]{BJoy}.

\section{Strategy of the proof}\label{Strategy}
In this section we will sketch the main ideas of the proof of Theorem \ref{ThmGeneralIneqs}. To this end, we believe it is in the benefit of the reader interested to consult \cite{BK} to keep the notation used there. \par We start by noting that, in light of the fact that $ \overline N(a,c,n)=\overline N(c-a,c,n), $ property which can be easily deduced from $ \ol N(m,n)=\ol N(-m,n) $ (see, e.g., \cite[Proposition 1.1]{Lovejoy}), it is enough to restrict our attention to the case when $ 0\le a<b\le\left\lfloor\frac{c}{2}\right\rfloor, $ and this is why we work under this assumption.
\subsection{Orthogonality relation}
From the orthogonality of the roots of unity, it follows that
\begin{equation}\label{orthogonality}
\sum_{n=0}^{\infty}\overline N(a,c,n)q^n=\frac1c\sum_{n=0}^{\infty}\overline p(n)q^n+\frac1c\sum_{j=1}^{c-1}\zeta_c^{-aj}\cdot \O(\zeta_c^j;q).
\end{equation}
Combining Theorem \ref{MainThm} with the identity \eqref{orthogonality}, we thus have
\begin{equation}\label{DifferenceOddc}
\sum_{n=0}^{\infty}(\overline N(a,c,n)-\ol N(b,c,n))q^n=\frac2c\sum_{j=1}^{\frac{c-1}{2}}\rho_j(a,b,c)  \O(\zeta_c^j;q)
\end{equation} if $ 2\nmid c, $
and \begin{equation}\label{DifferenceEvenc}
\sum_{n=0}^{\infty}(\overline N(a,c,n)-\ol N(b,c,n))q^n=\frac2c\sum_{j=1}^{\frac{c-2}{2}}\rho_j(a,b,c)  \O(\zeta_c^j;q)+((-1)^a-(-1)^b)\O(-1;q).
\end{equation} if $ 2\mid c, $ where \[\rho_j(a,b,c):=\cos\left( \frac{2\pi a j}{c}\right)-\cos\left(\frac{2\pi bj}{c} \right) . \]
\par Before moving forward, a  remark is in order, namely that  in the summations carried over $ j $ in the identities  \eqref{orthogonality}--\eqref{DifferenceEvenc}, we need to replace $ j $ and $ c $ by $ j'=\frac j{(c,j)} $ and $ c'=\frac c{(c,j)}. $  As such, whenever $ (c,j)>1, $  the statement of Theorem \ref{MainThm} applies to the coefficients 
$A\big(\frac{j'}{c'};q\big).$
\subsection{The main asymptotical contributions} Let us only focus for now on the case when $ c $ is odd, as the other case will be treated in essentially the same manner. For simplicity, we can further assume that $ c $ is prime. As explained in \cite{BK}, and as it will also become clear in the course of our proof, this assumption does not restrict the generality and it will only lead to slightly different bounds when estimating the error terms (which we do explicitly in the next section), without affecting the main result whatsoever. Coming back to our problem, let us write
\begin{equation*}
\overline N(a,c,n)-\ol N(b,c,n)=\sum_{j=1}^{\frac{c-1}{2}}(S_j(a,b,c)+T_j(a,b,c))+O_c(n^{\varepsilon}),
\end{equation*}
where
\begin{equation*}
S_j(a,b,c):=2\rho_j(a,b,c)i\sqrt{\frac 2n}\sum_{\substack{1\le k\le\sqrt n\\c|k,~2\nmid k}}\frac{B_{j,c,k}(-n,0)}{c\sqrt k}\cdot\sinh\left( \frac{\pi\sqrt n}{k}\right)
\end{equation*}and
\begin{equation*}
T_j(a,b,c):=4\rho_j(a,b,c)\sqrt{\frac 2n}\sum_{\substack{1\le k\le \sqrt n\\c\nmid k,~2\nmid k,~c_1\ne4\\r\ge0,~ \delta_{c,k,r}>0}}\frac{D_{j,c,k}(-n,m_{j,c,k,r})}{c\sqrt k}\cdot \sinh \left( \frac{4\pi\sqrt{\delta_{c,k,r}n}}{k} \right).
\end{equation*}  
\par Now, if we can identify the term that gives the main contribution to the sums $ S_j $ and $ T_j, $ and if this term has a positive coefficient, then we are done.
Indeed, we will shortly see that, for $ c\ge7, $ the main contribution can only come from the hyperbolic sine term corresponding to $ k=j=1$ and $r=0  $ in the summation for $ T_j$, in which case $ D_{1,c,1}=1 $ and $\rho_1(a,b,c)>0 $ for any $ 0\le a<b\le\frac{c-1}{2}.$ Proving that the value $ n_{a,b,c} $ exists, beyond which \eqref{GenericIneq} always holds, is then only a matter of carefully bounding some error terms. Our proof will also explain why, for the small moduli  $c\le 6$, the sign of the inequalities might depend on the residue class of $ n$ modulo $ c. $ 

\section{Proofs of the main results}\label{Proof}
In the way described in the previous section, we start by giving the proof of Theorem \ref{ThmGeneralIneqs}. The proof of Theorem \ref{ThmMod5} will then be nothing more than a computation, while Theorems \ref{ThmMod2}--\ref{ThmMod4} will follow fairly easily. 
\begin{proof}[Proof of Theorem \ref{ThmGeneralIneqs}]
We distinguish two cases, and we first treat the case when $ c\ge7 $ is odd.\\

\noindent {\sc Case 1:} $ 2\nmid c. $ Without restricting generality, assume that $ c $ is prime. While this assumption does not affect the result, it also simplifies our task, in that we do not have to worry about possible common divisors of $ c $ and $ j,  $ and so we do not need to pass down to $ j' $ or $ c' $ (if anything, this would just make the notation more cumbersome and lead to slightly different estimates for the error terms, but the conclusion is essentially the same). As already seen, we have 
\begin{equation*}
\overline N(a,c,n)-\ol N(b,c,n)=\sum_{j=1}^{\frac{c-1}{2}}(S_j(a,b,c)+T_j(a,b,c))+O_c(n^{\varepsilon}),
\end{equation*}
where
\begin{equation}
\label{Sj}
S_j(a,b,c)=2\rho_j(a,b,c)i\sqrt{\frac 2n}\sum_{\substack{1\le k\le\sqrt n\\c|k,~2\nmid k}}\frac{B_{j,c,k}(-n,0)}{c\sqrt k}\cdot\sinh\left( \frac{\pi\sqrt n}{k}\right)
\end{equation}and
\begin{equation}
\label{Tj}
T_j(a,b,c)=4\rho_j(a,b,c)\sqrt{\frac 2n}\sum_{\substack{1\le k\le \sqrt n\\c\nmid k,~2\nmid k,~c_1\ne4\\r\ge0,~ \delta_{c,k,r}>0}}\frac{D_{j,c,k}(-n,m_{j,c,k,r})}{c\sqrt k}\cdot \sinh \left( \frac{4\pi\sqrt{\delta_{c,k,r}n}}{k} \right).
\end{equation}
\subsection{Determining the dominant terms}
Knowing that the error term from Theorem \ref{MainThm} is of order $ O(n^\varepsilon), $ we only need to identify what is the main contribution coming from the sums $ S_j  $ and $ T_j. $ For this, we first compare the arguments of
the hyperbolic sines aapearing in \eqref{Sj} and \eqref{Tj}. In $ S_j, $ the main term occurs for $ k=c, $ giving $ \sinh\left(\frac{\pi\sqrt n}c\right) $ as the main hyperbolic sine argument, whereas in $ T_j $ it is given by $ k=1, $ in which case we have $ k_1=1 $ and $ c_1=c, $ thus the condition $ \ell\equiv jk_1\pmod*{c} $ yields $ \ell=j. $ Due to symmetry reasons, we can assume that $ \frac{\l}{c}\le\frac14, $ hence we compute  $ \delta_{c,k,r}=\delta_{c,1,r}=\big( \frac jc-\frac14\big)^2-r\frac{\l}{c_1},  $ with $ r\ge0, $ which is clearly maximized by $ \delta_{c,1,0}=\big( \frac jc-\frac14\big)^2,$ in case $ r=0. $
Running over  $ j, $ the maximum value is obviously attained for $ j=1, $ which gives $ \sinh \left( \pi\sqrt n\left(1-\frac 4c \right)\right)  $ as the main hyperbolic sine term. We only need to make sure that $ D_{1,c,1}(-n,0)\ne0, $ but this is clear since, for $ k=1, $  $ r=0 $ and any $ j, $ we easily see that $ D_{j,c,k}(-n,m_{j,c,k,r})=D_{j,c,1}(-n,0)=1. $  Now, since \[1-\frac4c=\frac{c-4}{c}>\frac1c\]
for any $ c\ge7 $ (in fact, for any $ c\ge6 $), we conclude that the term which gives the main contribution equals 
\begin{equation}\label{MainBound1}
T_1(a,b,c)=\frac4{c\sqrt n}\tan\left(\frac{\pi}c \right) \rho_1(a,b,c)\sinh \left( \pi\sqrt n\left(1-\frac 4c \right)\right),  \end{equation}
and this is positive for all $ 0\le a<b\le \frac{c-1}{2}. $
In order to infer that the inequality \eqref{GenericIneq} holds for $ n $ sufficiently large, we only need to prove that the contribution of the other terms entering the expressions of $ S_j $ and $ T_j $ is asymptotically smaller than $ T_1(a,b,c), $ and this is what we will do next.
\subsection{Bounding the contributions of $S_j$ and $T_j$. } We first consider $S_j$ and estimate
\begin{align*}
|S_j(a,b,c)|&\le\frac{2\sqrt2|\rho_j(a,b,c)|}{c\sqrt n}\sum_{\substack{1\le k\le\sqrt n\\c|k,~2\nmid k}}\frac{|B_{j,c,k}(-n,0)|}{\sqrt k}\cdot\sinh\left( \frac{\pi\sqrt n}{k}\right)\\
&\le\frac{2\sqrt2|\rho_j(a,b,c)|}{c\sqrt n}\left|\tan\left(\frac{\pi j}{c} \right)  \right|  \sinh\left( \frac{\pi\sqrt n}{c}\right)\sum_{\substack{1\le k\le\sqrt n\\c|k,~2\nmid k}}k^{-\frac12}\primesum\sum_{h}\frac1{\sin\left(\frac{\pi h}{c} \right) }
\end{align*}
We bound the innermost sum using the estimate 
\[\primesum\sum_{h}\frac1{\left| \sin\left(\frac{\pi h}{c} \right)\right| }\le\frac{2k}c\sum_{h=1}^{\frac{c-1}{2}}\frac1{\left| \sin\left(\frac{\pi h}{c} \right)\right|}\le \frac{2k}{\pi} \sum_{h=1}^{\frac{c-1}{2}}\frac1{h\left(1-\frac{\pi^2}{24} \right) }\le \frac{2k\left( 1+\log\left(\frac{c-1}{2} \right) \right) }{\pi\left( 1-\frac{\pi^2}{24}\right) },\]
where the first inequality follows from the fact that $ c\le k $ (because by assumption $ c\mid k $), the second from the inequality $ |\sin x|\ge x-\frac{x^3}{6}, $ and the third from the well-known bound for the harmonic series.
This gives further 
\begin{align*}
|S_j(a,b,c)|&<\frac{4\sqrt2|\rho_j(a,b,c)|}{c\sqrt n}\frac{\big|\tan\big(\frac{\pi j}{c} \big)  \big| \left( 1+\log\left(\frac{c-1}{2} \right)\right) }{\pi\left( 1-\frac{\pi^2}{24}\right)} \sinh\left( \frac{\pi\sqrt n}{c}\right)\sum_{\substack{1\le k\le\sqrt n\\c|k,~2\nmid k}}\sqrt k\\&
<\frac{16\cot\left(\frac{\pi }{2c} \right)   \left( 1+\log\left(\frac{c-1}{2} \right) \right) }{c^2\pi\left( 1-\frac{\pi^2}{24}\right)}n^{\frac34} \sinh\left( \frac{\pi\sqrt n}{c}\right),
\end{align*}where we used the Cauchy-Schwarz inequality to bound the sum over $ k $.
\par We next explicitly estimate the error coming from $ T_j. $ We have to look at the terms with $ k\ge2 $ and at those that  have $ k=1, $ but are different than the main term. For this, we trivially bound  $$ |D_{j,c,k}| \le\frac k{\sqrt2}\tan\left(\frac{\pi j}{c} \right). $$ For $ k\ge2, $ the argument of the hyperbolic sine is at most half that of the main term, namely  
\[\sinh \left( \pi\sqrt n\left(\frac12-\frac 2c \right)\right).\]
Next, note that the number of $ r $'s satisfying $ \delta_{c,k,r}>0 $ is decreasing as a function of $ \ell, $ and thus attains its maximum at $ \ell=1, $ in which case it equals
\[\left\lfloor\frac{c_1}{16}-\frac12+\frac1{c_1}\right\rfloor<\frac{c+8}{16},\]
therefore the contribution coming from $ k\ge2 $ can be estimated against 
\begin{equation}\label{MainBound2}
\frac{c+8}{2c\sqrt n}n^{\frac34}\cot\left(\frac{\pi }{2c} \right)\sinh \left( \pi\sqrt n\left(\frac12-\frac 2c \right)\right)<\frac{c+8}{4c\sqrt n}n^{\frac34}\cot\left(\frac{\pi }{2c} \right)e^{\pi\sqrt n\left(\frac12-\frac 2c \right)}.\end{equation}
Finally, the terms with $ k=1 $ have $ j\ge2. $ But $ k=1 $ implies $ k_1=1, $ and so $ \l=j\ge2, $ which then means that the contribution coming from these terms can be estimated against 
\begin{equation}\label{MainBound3}\frac{c+8}{2c\sqrt n}\cot\left(\frac{\pi }{2c} \right)\sinh \left( \pi\sqrt n\left(1-\frac 8c \right)\right)<\frac{c+8}{4c\sqrt n}\cot\left(\frac{\pi }{2c} \right)e^{\pi\sqrt n\left(1-\frac 8c \right)}.\end{equation}
On comparing \eqref{MainBound2} and \eqref{MainBound3} with the main term $ T_1(a,b,c) $ from \eqref{MainBound1}, we conclude this part of the proof. What is left to do is to make the error term explicit and bound it in an optimal way.  We do so after discussing the case $ 2\mid c. $\\

\noindent{\sc Case 2:} $ 2\mid c. $ In this case, we have 
\begin{equation*}
\sum_{n=0}^{\infty}(\overline N(a,c,n)-\ol N(b,c,n))q^n=\frac2c\sum_{j=1}^{\frac{c-2}{2}}\rho_j(a,b,c)  \O(\zeta_c^j;q)+((-1)^a-(-1)^b)\O(-1;q).
\end{equation*}
As explained in the proof of Corollary 1 from \cite[p. 22]{Ciolan}, the coefficients of $ \O(-1;q), $  a harmonic Maass form of weight $3/2,$ are of order $ O(n^\varepsilon). $ Therefore, we can ignore them from our analysis, since the main contribution will come from the sum over $j.$
\par We continue by noting that, since $ c $ is even, the terms $ B_{j,c,k} $ do not contribute to the sum, hence
\begin{equation*}
\sum_{n=0}^{\infty}(\overline N(a,c,n)-\ol N(b,c,n))q^n=\frac2c\sum_{j=1}^{\frac{c-2}{2}}T_j(a,b,c)+O(n^\varepsilon),
\end{equation*}
where 
\[T_j(a,b,c)=4\rho_j(a,b,c)\sqrt{\frac 2n}\sum_{\substack{1\le k\le \sqrt n\\c\nmid k,~2\nmid k,~c_1\ne4\\r\ge0,~ \delta_{c,k,r}>0}}\frac{D_{j,c,k}(-n,m_{j,c,k,r})}{c\sqrt k}\cdot \sinh \left( \frac{4\pi\sqrt{\delta_{c,k,r}n}}{k} \right).\]
The term corresponding to $ j=1 $ and $ k=1 $ gives $\sinh \left( \pi\sqrt n\left(1-\frac 4c \right)\right)  $ as maximum argument. For any other $ j $ (which might now have common divisors with $ c $), set $ j'=\frac{j}{(c,j)} $ and $ c'=\frac{c}{(c,j)}. $ From the congruence $ j'k_1\equiv\l\pmod*{c_1'} $ we see that any value $ c_1'\le 4 $ will give $ \delta_{c',k,r}=0, $ hence there is no hyperbolic sine argument contributing. For $ c_1'>4, $ it is then easy to see that $ \delta_{c',k,r} $ is maximized by $ k=1, $ which in turn forces $ k_1=1 $ and $ c_1'=c'. $ The hyperbolic sine term will then be given by $ \sinh \left( \pi\sqrt n\left(1-\frac {4j'}{c'} \right)\right)=\sinh \left( \pi\sqrt n\left(1-\frac {4j}{c} \right)\right),   $ and this is smaller than the leading term for any $ j\ge2. $ 
\par We can now conclude the proof in this case by arguing in the following way. Pick any $ j\ge2. $ If $ c' $ is even, then we are done, as the maximum possible contribution would be $$ \sinh \left(\pi\sqrt n \left(1-\frac{4j'}{c_1'} \right) \right)=\sinh \left(\pi\sqrt n \left(1-\frac{4j'}{c'} \right) \right)=\sinh \left(\pi\sqrt n \left(1-\frac{4j}{c} \right) \right)\le \sinh \left(\pi\sqrt n \left(1-\frac{8}{c} \right) \right),$$ which is smaller than the main term. The only issue might appear when $ 2\nmid c', $ as there may be then another possible main term coming from the $ B_{j,c,k} $ sum, namely the one containing $ \sinh\left( \frac{\pi\sqrt n}{c'}\right). $ In order to establish which one is bigger between these two, we need to compare $ \frac1{c'} $ with $ 1-\frac{4}{c}. $ Since we have \[\frac{1}{c'}+\frac4c=\frac jc+\frac4c\le\frac12+\frac12=1\] for $ c\ge8, $ we are done. 
\par What this shows is that for any $ c\ge7, $ regardless of whether $ c $ is even or odd, the main contribution is always given by the term in \eqref{MainBound1}. A separate study is required for $ c\le 6, $ as then the hyperbolic sine arguments from the sum $B_{j,c,k}$ and $ D_{j,c,k} $ might coincide, and we need to explicitly evaluate their coefficients in order to establish which one dominates asymptotically. As the case $c=6 $  has already been covered, one only needs to study the cases $ c\in\{2,3,4,5\}. $
\par Next, we want to make explicit the error term appearing in Theorem \ref{MainThm}. This would ensure then the existence of a number $ n_{a,b,c} $ such that the inequalities hold for every $ n>n_{a,b,c}. $
\subsection{Estimating the error terms} The analysis is a bit tedious, as the various sums give different types of errors and, as such, we need to bound them in different ways. In some cases, the arguments are similar to those from \cite{BK},  while in some others we manage to simplify them and come up with neater estimates. Certainly, and this might be an interesting question for a minutious reader, our bounds can be improved. 
\par As we do not want to repeat too much material from \cite{Ciolan}, we kindly invite the reader to follow the steps presented in the proof of  \cite[Theorem 1]{Ciolan}. This is essential in understanding the following estimates.
\subsubsection{Estimation of the error term arising from the Circle Method} 
For $ \sum_2, $ the error term, let us denote it by $ S_2, $ comes from the two sums taken over $ r\ge1 $ in the expression for $ \widetilde \U\left(\frac{ah'}{c};q_1 \right)  $ from \cite[p. 15]{Ciolan}.
Using well-known facts from the theory of Farey arcs, such as
\begin{equation*}
{\rm Re}(z)=\frac kn,\quad  {\rm Re}\left( \frac 1z\right)>\frac k2,\quad |z|^{-\frac12}\le n^{\frac12}\cdot k^{-\frac12} \quad\text{and}\quad \vartheta_{h,k}'+\vartheta_{h,k}''\le \frac{2}{k(N+1)},
\end{equation*} 
we can bound the error term coming from $\sum_2$ by
\begin{align*}S_2&< 4e^{2\pi}\sqrt2\cot\left(\frac{\pi }{2c} \right)\sum_{k}k^{-\frac12}\cdot \sum_{r\ge1} \ol p(r)\Big( e^{-\frac{(16r-1)\pi}{16}}+e^{-\frac{(16r+7)\pi}{16}}\Big)\\&<4c_1e^{2\pi}\sqrt2\cot\left(\frac{\pi }{2c} \right)\sum_{k}k^{-\frac12}, \end{align*}
where $$ c_1:=\sum_{r\ge1} e^{\pi\sqrt r}\Big( e^{-\frac{(16r-1)\pi}{16}}+e^{-\frac{(16r+7)\pi}{16}}\Big). $$
In the same way we bound the error terms coming from the sums $ \sum_5 $ and $ \sum_6, $ which we denote $ S_5 $ and $ S_6, $  by
\begin{align*}S_5&<c_2e^{2\pi}\sqrt2\cot\left(\frac{\pi }{2c} \right)\sum_{k}k^{-\frac12},\\
S_6&<\frac{c_2e^{2\pi}}{\sqrt2}\cot\left(\frac{\pi }{2c} \right)\sum_{k}k^{-\frac12},\end{align*}where
\[c_2:=2\sum_{r\ge1}e^{\pi\sqrt r}e^{-\frac{(c^2-8)\pi r}{16c^2}}.\]
Now, for the sums $ \sum_1,\sum_3,\sum_4 $ we can argue similarly, and, denoting the error terms by $ S_1,S_3,S_4 $ respectively, we obtain 
\begin{align*}
S_1&<4c_3e^{2\pi}\cot\left(\frac{\pi }{2c} \right)\sum_{k}k^{-\frac12},\\
S_3&<2c_4e^{2\pi}\cot\left(\frac{\pi }{2c} \right)\sum_{k}k^{-\frac12},\\
S_4&<c_5e^{2\pi}\cot\left(\frac{\pi }{2c} \right)\sum_{k}k^{-\frac12},
\end{align*}
where \begin{align*}
c_3&:=\sum_{r\ge1}e^{\pi\sqrt r}e^{-\pi r},\\
c_4&:=\sum_{r\ge1}e^{\pi\sqrt r}e^{-\frac{\pi r}{2c^2}},\\
c_5&:=\sum_{r\ge1}e^{\pi\sqrt r}e^{-\frac{(2r+1)\pi }{8}}.
\end{align*}

\subsubsection{Error terms given by the Mordell integrals} Note that $ H_{j,c}(x)=H_{j,c}^+(x)+H_{j,c}^-(x), $
where \[H_{j,c}^{\pm}=\pm\frac{i}{8\cosh\left(\frac x2 \right)\sin \left(\frac{\pi a}{c} \right)\sinh\left(\frac x2\pm\frac{\pi a }{c} \right)   }.\]
We can therefore split the sum just like in \cite[pp. 937--938]{BK}, and we denote the contributions of these functions to $ I_{j,c,k,\nu} $ by $ I_{j,c,k,\nu}^\pm. $ From the proof of \cite[Lemma 1]{Ciolan} we obtain 
\begin{align*}z^{\frac12}I_{j,c,k,\nu}^\pm&\le \frac{\sqrt k}{8\sqrt2\big| \sin\big(\frac{\pi\nu}{k} -\frac{\pi}{4k}\pm\frac{\pi j}{c}\big)\big| \big| \sin \big(\frac{\pi j}{c}\big)\big|  \big( \RE\big(\frac1z \big) |z|\big) ^{\frac12}}\\&\le\frac{\sqrt n}{8\sqrt k\big| \sin\big(\frac{\pi\nu}{k} -\frac{\pi}{4k}+\frac{\pi a}{c}\big)\big| \big| \sin \big(\frac{\pi j}{c}\big)\big| }.\end{align*}
Denoting 
\[\cal S_{a,c,k}^\pm:=\sum_{\nu=1}^{k}\frac1{\left| \sin\left(\frac{\pi \nu}{k}-\frac{\pi }{4k}\pm\frac{\pi  a}{c} \right)\right|},\]
we have \begin{align*}
\cal S_{a,c,k}^-&\le \sum_{\nu=1}^{k}\frac{1}{\left| \sin\left(\frac{\pi\nu}k-\frac {\pi}{2k} \right)\right|}=\sum_{\nu=1}^{\left\lfloor\frac k2\right\rfloor}\frac1{ \sin\left(\frac{\pi\nu}k-\frac{\pi}{2k} \right)}+\sum_{\nu=0}^{\left\lfloor\frac{k+1}2\right\rfloor-1}\frac1{ \sin\left(\frac{\pi\nu}k+\frac{\pi}{2k} \right)}\\
&<\frac k{\pi} \sum_{\nu=1}^{\left\lfloor\frac k2\right\rfloor}\frac1{\left(\nu-\frac12\right)\left(1-\frac16\left( \frac{\pi}{k}\left( \left\lfloor\frac k2 \right\rfloor-\frac12\right) \right)^2  \right)  }+\frac k{\pi}\sum_{\nu=0}^{\left\lfloor\frac {k+1}2\right\rfloor-1}\frac1{\left(\nu+\frac12\right)\left(1-\frac16\left( \frac{\pi}{k}\left( \left\lfloor\frac{k+1}2 \right\rfloor-\frac12\right) \right)^2  \right)  } 
\\
&<\frac{2k\log\left(\frac k2 \right) }{\pi\left(1-\frac{\pi^2}{24} \right) },
\end{align*}
and similarly for $ \cal S_{a,c,k}^+.$ We thus obtain 
\[\sum_7 \le \frac{\log\left(\frac k2 \right) }{2\pi\left(1-\frac{\pi^2}{24} \right)\sin\left(\frac{\pi}{c} \right)} \sum_k \sqrt k<\frac{n^{\frac34}\log \left( \frac n4\right) }{2\pi\left(1-\frac{\pi^2}{24} \right)\sin\left(\frac{\pi}{c} \right) },\] and a similar bound holds for $ \sum_8. $
\subsubsection{Symmetrizing the paths of integration} Such errors are only given by the sums that contribute to the main term, which are $ \sum_2,\sum_5$ and $\sum_6. $ As it can be seen in \cite[p. 19]{Ciolan}, making the path of integration symmetric leads to a Hankel-type integral that will give the main term, and an error term arising from integrating over the remaining parts of the interval. This error is what we estimate in what follows.
\par We decompose
\[\int_{-\vartheta_{h,k}'}^{\vartheta_{h,k}''}=\int_{-\frac1{kN}}^{\frac1{kN}}-\int_{-\frac1{kN}}^{-\frac1{k(k_1+k)}}-\int_{\frac1{k(k_2+k)} }^{\frac1{kN}}\]
and we want to  estimate the contributions to the error terms from the last two integrals in the same way as before, with the only difference that, on these other parts of the Farey intervals, we have
$$ {\rm Re}(z)=\frac kn, \quad {\rm Re}\left(\frac 1z \right)<k \quad \text{and}\quad  |z|^2\ge \frac{k^2}{n^2}. $$
In this way, we get that this contribution is less than 
\[\sqrt 2e^{2\pi+\frac{\pi}{8}}\cot\left(\frac{\pi}{2c} \right)n^{-\frac12} \frac{ 1+\log\left(\frac{c-1}{2}  \right) }{\pi\left( 1-\frac{\pi^2}{24}\right) }\sum k^{-\frac12}. \]
Now, since $ \delta_{c,k,r}\le\frac1{16}, $ the exact same bounds (multiplied by 2 for $ \sum_5 $) hold for $ \sum_5 $ and $ \sum_6. $ 
\subsubsection{Errors introduced by integrating along the smaller arc} This is very easy and goes along the same lines as in \cite[p. 939]{BK}. The contribution to $ \sum_2 $ is less than  

\[2\sqrt2\left(\frac43+2^{\frac54}\right)e^{2\pi+\frac{\pi}8} \cot\left(\frac{\pi}{2c} \right)\frac{1+\log\left(\frac{c-1}{2}  \right) }{\pi\left( 1-\frac{\pi^2}{24}\right) }n^{\frac14}, \]
the precise same bounds (multiplied by 2 for $ \sum_5 $) being valid for $ \sum_5 $ and $ \sum_6. $ 
Putting together all these bounds, the proof is complete.
\end{proof}
More than the result in itself, the following is a perfect example of Theorem \ref{ThmGeneralIneqs} at play and it illustrates why the general reasoning fails for moduli $ c<7. $
\begin{proof}[Proof of Theorem \ref{ThmMod5}]
As it can be readily seen, the argument of the leading hyperbolic sine term is the same in the sums $ B_{1,5,5}, $ $B_{2,5,5}$ and $ D_{1,5,1}, $ and equals $ \sinh\left( \frac{\pi\sqrt{n}}{5}\right) . $ Therefore, to establish the sign of the rank difference $ \ol N(a,5,n)-\ol N(b,5,n), $ we need to compute the leading coefficients of these  terms. Without much effort, and making use of several known properties of the Dedekind sums (see, e.g., \cite[Ch. 3.7]{Apostol}), we compute
\[i\sqrt{\frac2n}B_{1,5,5}(-n,0)={\frac{1}{\sqrt{5n}}}\tan\left( \frac{\pi}{5}\right) \cdot\begin{cases}
\phantom{+}0 & \text{if~$n\equiv0\pmod*5,$}\\
\phantom{+}3+\sqrt5 & \text{if~$n\equiv1\pmod*5,$}\\
\phantom{+}1-\sqrt5 & \text{if~$n\equiv2\pmod*5,$}\\
-1-\sqrt5 & \text{if~$n\equiv3\pmod*5,$}\\
-3+\sqrt5 & \text{if~$n\equiv4\pmod*5,$}\\
\end{cases}\]
and \[i\sqrt{\frac2n}B_{2,5,5}(-n,0)={\frac{1}{\sqrt{5n}}}\tan\left( \frac{2\pi}{5}\right) \cdot\begin{cases}
\phantom{+}0 & \text{if~$n\equiv0\pmod*5,$}\\
\phantom{+}4 & \text{if~$n\equiv1\pmod*5,$}\\
-2 & \text{if~$n\equiv2\pmod*5,$}\\
\phantom{+}2 & \text{if~$n\equiv3\pmod*5,$}\\
-4 & \text{if~$n\equiv4\pmod*5,$}\\
\end{cases}\]
while 
\[2\sqrt{\frac2n}D_{1,5,1}(-n,0)=\frac{2}{\sqrt n}\tan\left( \frac{\pi}{5}\right).\]
Bounding the error terms as described before, a numerical check in Mathematica shows that the result holds for all values $ n>40, $ with the exceptions presented in Table \ref{TableMod5}.\par One should also note that, although our approach cannot establish identities, a simple trigonometric computation shows  that $$\left(1+\frac{1}{\sqrt5} \right) \left(1-\cos\left( \frac{4\pi}{5}\right) \right)\tan\left(\frac{\pi}{5} \right)+ \frac1{\sqrt5}\left(1-\cos\left( \frac{8\pi}{5}\right)\right) \tan\left( \frac{2\pi}{5}\right)=0,  $$   which means that the main asymptotic contributions of $ \ol N(a,5,5n+d)$ and $ \ol N(b,5,5n+d)$ coincide and, as such, is an indication of the fact that $ \ol N(a,5,5n+d) = \ol N(b,5,5n+d)$ holds $ a=0,~b=2 $ and $ d=2. $ This result was already proven by Lovejoy and Osburn, see eq. (12) from \cite[Theorem 1.2]{LO}.
\end{proof}

\begin{rem}
For $ c=6 $ we can also give an alternative proof to all the inequalities listed in Theorem 1.2 from \cite{JZZ}. Indeed, we have
\begin{align*}\sum_{n=0}^{\infty}(\overline N(a,6,n)-\ol N(b,6,n))q^n&=\frac13\rho_1(a,b,6)  \O\left(\frac16;q\right)+\frac13\rho_2(a,b,6)\O\left(\frac16;q\right)+\frac16 \O\left(\frac12;q\right)\\
&=\frac13\rho_1(a,b,6)  \O\left(\frac16;q\right)+\frac13\rho_2(a,b,6)\O\left(\frac16;q\right)+O(n^\varepsilon)
\end{align*}
The main contribution of the first term equals  $$\rho_1(a,b,6)A\left(\frac1{6};n \right)\sim \frac4{c\sqrt n}\tan\left(\frac{\pi}{3} \right)\left( \cos\left(\frac{\pi a}{3}\right) -\cos\left(\frac{\pi b}{3} \right) \right)  \sinh\left( \frac{\pi \sqrt n}{3}\right),  $$
while that of the second term equals
\[\rho_2(a,b,6)A\left(\frac1{3};n \right)\sim\begin{cases}\phantom{-}\dfrac2{3\sqrt{n}}\tan\left( \dfrac{\pi}{3}\right)\left( \cos\left(\frac{2\pi a}{3}\right) -\cos\left(\frac{2\pi b}{3} \right) \right)\sinh \left( \dfrac{\pi\sqrt{n}}{3} \right)&\text{if~}n\equiv0,1\pmod*3,\\
-\dfrac4{3\sqrt{n}}\tan\left( \dfrac{\pi}{3}\right)\left( \cos\left(\frac{2\pi a}{3}\right) -\cos\left(\frac{2\pi b}{3} \right) \right)\sinh \left( \dfrac{\pi\sqrt{n}}{3} \right)&\text{if~}n\equiv2\phantom{,0}\pmod*3. \end{cases}\]
Since, for $ j=2, $ the arguments appearing in $ \rho_2(a,b,6) $ are no longer in the interval $ [0,\pi), $ on which the cosine function is decreasing, the inequalities now depend, additionally, on the values of $ a $ and $ b. $ While our line of reasoning cannot be used to prove the identities from \cite[Theorem 1.4]{JZZ}, it clearly suggests that they should hold true, in light of the fact that the main hyperbolic sines are equal.
In passing, we also correct Example 1 from \cite{Ciolan}, in which some misprints seem to have occurred, the correct asymptotic values for $ A\left(\frac1{3};n \right) $ being those given here. In particular, Theorem \ref{ThmMod3} is now a straightforward consequence of these asymptotics. 
\end{rem}
Further, the proof of Theorem \ref{ThmMod2} follows on noting that, from identity \eqref{orthogonality}, we obtain
\[N(0,2,n)-N(1,2,n)=\O(-1;q)\]
and on invoking the next result. The numerical check is in this case greatly simplified.
\begin{Lem}
The coefficients (other than the leading term) of the series $$ \O(-1;q)=1+2q-4q^2+8q^3-10q^4+\cdots $$ are alternating in sign. What is the same,  the coefficients (other than the leading term) of the series  $$ \O(-1;-q)=1-2q-4q^2-8q^3-10q^4-\cdots $$ are negative.
\end{Lem}
\begin{proof}
Summing according to the largest part of the overpartition, we obtain, cf. eq. (3.1) from \cite{Lovejoy},
\begin{equation}
\label{O-1-q}
\O(z;q)=1+2q+z^{-1}\sum_{n=1}^{\infty}\frac{(-q/z)_n(zq)^n(1+zq)}{(q/z)_n}.
\end{equation}
Setting $ z=-1 $ and $ q\mapsto -q $ in \eqref{O-1-q}, we have
\[\O(-1;-q)=1-2q-\sum_{n=1}^{\infty}\frac{(-q)_nq^n(1+q)}{(q)_n}=1-2q-\sum_{n=1}^{\infty}\prod_{k=1}^n\frac{1+q^k}{1-q^k}(1+q)q^n,\]
and it is clear that the coefficients of the sum on the right, expressed as a $ q $-series, are all positive.
\end{proof}
Using the fact (see, e.g., the
proof of Theorem 5.6 from \cite[p. 330]{Lovejoy}) that \[\O(i;q)=1+2\sum_{n=1}q^{n^2},\] Theorem \ref{ThmMod4} becomes an easy exercise, which we leave to the interested reader (this was also proven, by a different and independent approach, in \cite{CGS}).
\section*{Acknowledgments}
The paper was completed during a stay at the Max Planck Instute for Mathematics, Bonn. The author is grateful to the institute and its staff for their hospitality and support. 

\end{document}